\documentclass[a4paper, 11pt, reqno]{smfart}

\usepackage{graphicx}
\usepackage{etex}
\usepackage{pifont}
\usepackage{amssymb,amsmath,mathrsfs,graphicx,float,mathptm}
\usepackage[T1]{fontenc}
\usepackage[all]{xy}
\usepackage[english,francais]{babel}
\usepackage{hyperref}
\usepackage[all]{xy}

\theoremstyle{plain}
\newtheorem{thm}{Theorem}[section]
\newtheorem{pro}[thm]{Proposition}
\newtheorem{lem}[thm]{Lemma}
\newtheorem{cor}[thm]{Corollary}

\newtheorem{theoalph}{Theorem}

\theoremstyle{definition}
\newtheorem{defi}[thm]{Definition}

\newtheorem{rem}[thm]{Remark}

\def\og{\leavevmode\raise.3ex\hbox{$\scriptscriptstyle\langle\!\langle$~}}
\def\fg{\leavevmode\raise.3ex\hbox{~$\!\scriptscriptstyle\,\rangle\!\rangle$}}

\addtocounter{section}{0}             % Start with section 1
\numberwithin{equation}{section}       % Number formulas within sections

\begin{document}
\selectlanguage{english}

\title[Action of the Cremona group on foliations on $\mathbb{P}^2_\mathbb{C}$: some curious facts]{Action of the Cremona group on foliations on $\mathbb{P}^2_\mathbb{C}$: \\ some curious facts}
\thanks{Second author supported by the Swiss National Science Foundation grant no 
PP00P2\_128422 /1 and by ANR Grant "BirPol"  ANR-11-JS01-004-01.}

\author{Dominique \textsc{Cerveau}}

\address{Membre de l'Institut Universitaire de France.
IRMAR, UMR 6625 du CNRS, Universit\'e de Rennes $1$, $35042$ Rennes, France.}
\email{dominique.cerveau@univ-rennes1.fr}

\author{Julie \textsc{D\'eserti}}
\address{IMJ-PRG, UMR $7586$, Universit\'e 
Paris $7$, B\^atiment Sophie Germain, Case $7012$, $75205$ Paris Cedex 13, France.}
\email{julie.deserti@imj-prg.fr}

\maketitle

\begin{abstract}
The Cremona group of birational transformations of $\mathbb{P}^2_\mathbb{C}$ acts on the space $\mathbb{F}(2)$ of holomorphic foliations on the complex projective plane. Since this action is not compatible with the natural graduation of $\mathbb{F}(2)$ by the degree, its description is complicated. The fixed points of the action are essentially described by Cantat-Favre in \cite{CF}. In that paper we are interested in problems of "aberration of the degree" that is pairs $(\phi,\mathcal{F})\in\mathrm{Bir}(\mathbb{P}^2_\mathbb{C})\times\mathbb{F}(2)$ for which $\deg\phi^*\mathcal{F}<(\deg\mathcal{F}+1)\deg\phi+\deg\phi-2$, the generic degree of such pull-back. We introduce the notion of numerical invariance ($\deg\phi^*\mathcal{F}=\deg\mathcal{F}$) and relate it in small degrees to the existence of transversal structure for the considered foliations.

\end{abstract}

\section{Introduction}

Let us consider on the complex projective plane $\mathbb{P}^2_\mathbb{C}$ a foliation $\mathcal{F}$ of degree $d$ and a birational map $\phi$ of degree $k$. If the pair $(\mathcal{F},\phi)$ is generic then 
\[
\deg\phi^*\mathcal{F}=(d+1)k+k-2. 
\]
For example if $\mathcal{F}$ and $\phi$ are both of degree $2$, then $\phi^*\mathcal{F}$ is of degree $6$. Nevertheless one has the following statement which says that "aberration of the degree" is not exceptional:

\begin{theoalph}\label{Thm:main}
For any foliation $\mathcal{F}$ of degree $2$ on $\mathbb{P}^2_\mathbb{C}$ there exists a quadratic birational map $\psi$ of~$\mathbb{P}^2_\mathbb{C}$ such that $\deg\psi^*\mathcal{F}\leq 3.$
\end{theoalph}

Holomorphic singular foliations on compact complex projective surfaces have been classified up to birational equivalence by Brunella, McQuillan and Mendes (\cite{Brunella}). Let $\mathcal{F}$ be a holomorphic singular foliation on a compact complex projective surface $S$. Let $\mathrm{Bir}(\mathcal{F})$ (resp. $\mathrm{Aut}(\mathcal{F})$) denote the group of birational (resp. biholomorphic) maps of $S$ that send leaf to leaf. If $\mathcal{F}$ is of general type, then $\mathrm{Bir}(\mathcal{F})=\mathrm{Aut}(\mathcal{F})$ is a finite group. In \cite{CF} Cantat and Favre classify the pairs $(S,\mathcal{F})$ for which $\mathrm{Bir}(\mathcal{F})$ (resp. $\mathrm{Aut}(\mathcal{F})$) is infinite; in the case of $\mathbb{P}^2_\mathbb{C}$ such foliations are given by closed rational $1$-forms.

In this article we introduce a weaker notion: the numerical invariance. We consider on $\mathbb{P}^2_\mathbb{C}$ a pair $(\mathcal{F},\phi)$ of a foliation $\mathcal{F}$ of degree $d$ and a birational map $\phi$ of degree $k\geq 2$. The foliation $\mathcal{F}$ is \textbf{\textit{numerically invariant}} under the action of $\phi$ if $\deg\phi^*\mathcal{F}=\deg\mathcal{F}$. We characterize such pairs $(\mathcal{F},\phi)$ with $\deg\mathcal{F}=\deg\phi=2$ which is the first degree with deep (algebraic and dynamical) phenomena, both for foliations and birational maps. We prove that a numerically invariant foliation under the action of a generic quadratic map is special:

\begin{theoalph}\label{Thm:gen}
Let $\mathcal{F}$ be a foliation of degree $2$ on $\mathbb{P}^2_\mathbb{C}$ numerically invariant under the action of a generic quadratic birational map of $\mathbb{P}^2_\mathbb{C}$. Then $\mathcal{F}$ is transversely projective.
\end{theoalph}

In that statement generic means outside an hypersurface in the space $\mathring{\mathrm{B}}\mathrm{ir}_2$ of quadratic birational maps of~$\mathbb{P}^2_\mathbb{C}$.

For any quadratic birational map $\phi$ of $\mathbb{P}^2_\mathbb{C}$ there exists at least one foliation of degree $2$ on $\mathbb{P}^2_\mathbb{C}$ numerically invariant under the action of $\phi$ and we give "normal forms" for such foliations. We don't know if the foliations numerically invariant under the action of a non-generic quadratic birational map have a special transversal structure. Problem: for any birational map $\phi$ of degree $d\geq 3$, does there exist a foliation numerically invariant under the action of $\phi$ ?

A foliation $\mathcal{F}$ on $\mathbb{P}^2_\mathbb{C}$ is
\textbf{\textit{primitive}} if $\deg\mathcal{F}\leq\deg\phi^*\mathcal{F}$ for
any birational map~$\phi$. Foliations of degree~$0$ and $1$ are defined by a
rational closed $1$-form (it is a well-known fact, see for
example~\cite{CCD}). Hence a non-primitive foliation of degree $2$ is also
defined by a closed $1$-form that is a very special case of transversely
projective foliations. Generically a foliation of degree $2$ is
primitive. Remark that there are foliations that are pull-back by a rational
map of degree greater than $1$, and that are nevertheless primitive. This is
the case of the foliation given by $Q_1\,\mathrm{d}Q_2-Q_2\,\mathrm{d}Q_1$ where $Q_1$ and $Q_2$
denote two generic polynomials of degree $3$, in other words a generic pencil
of elliptic curves. The following problem seems relevant: classify in any degree the primitive foliations numerically invariant under the action of birational maps of degree~$\geq~2$; are such foliations transversely projective or is this situation specific to the degree~$2$~? In this vein we get the following statement.

\begin{theoalph}\label{Thm:degre3}
A foliation $\mathcal{F}$ of degree $2$ on $\mathbb{P}^2_\mathbb{C}$ numerically invariant under the action of a generic cubic birational map of $\mathbb{P}^2_\mathbb{C}$ satisfies the following properties:
\smallskip
\begin{itemize}
\item[$\bullet$] $\mathcal{F}$ is given by a closed rational $1$-form $($Liouvillian integrability$)$;
\smallskip
\item[$\bullet$] $\mathcal{F}$ is non-primitive.
\end{itemize}
\end{theoalph}

Is it a general fact, {\it i.e.} if $\mathcal{F}$ is numerically invariant under the action of $\phi$ and $\deg\phi\gg\deg\mathcal{F}$ is $\mathcal{F}$ Liouvillian integrable ?

\bigskip

The text is organized as follows: we first give some definitions, notations
and properties of birational maps of $\mathbb{P}^2_\mathbb{C}$ and foliations
on $\mathbb{P}^2_\mathbb{C}$. In \S\ref{Sec:proofthmA} we give a proof of
Theorem \ref{Thm:main}; we focus on foliations of degree $2$ on
$\mathbb{P}^2_\mathbb{C}$ that have at least two singular points, and then on
foliations of degree $2$ on $\mathbb{P}^2_\mathbb{C}$ with exactly one
singular point. The section \ref{Sec:numinv} is devoted to the description of
foliations of degree $2$ on $\mathbb{P}^2_\mathbb{C}$ numerically invariant
under the action of any quadratic birational map. This allows us to prove
Theorem \ref{Thm:gen}. At the end of the paper, \S\ref{Sec:higherdegree}, we
describe the foliations of degree $2$ numerically invariant under some cubic
birational maps of $\mathbb{P}^2_\mathbb{C}$, and we finally establish Theorem \ref{Thm:degre3}.

\subsection*{Acknowledgment} We thank Alcides Lins Neto for helpful discussions, and the anonymous referee for remarks and suggestions.

\section{Some definitions, notations and properties}

\subsection{About birational maps of $\mathbb{P}^2_\mathbb{C}$}

A \textbf{\textit{rational map}} $\phi$ of $\mathbb{P}^2_\mathbb{C}$ is a "map" of the type 
\[ 
\phi\colon\mathbb{P}^2_\mathbb{C}\dashrightarrow\mathbb{P}^2_\mathbb{C}, \quad\quad\quad (x:y:z)\dashrightarrow (\phi_0(x,y,z):\phi_1(x,y,z):\phi_2(x,y,z))
\]
where the $\phi_i$'s are homogeneous polynomials of the same degree and without common factor. The \textbf{\textit{degree}} of~$\phi$ is by definition the degree of the $\phi_i$'s. A \textbf{\textit{birational map}} $\phi$ of $\mathbb{P}^2_\mathbb{C}$ is a rational map having a rational "inverse" $\psi$, {\it i.e.} $\phi\circ\psi=\psi\circ\phi=\mathrm{id}$. The first examples are the birational maps of degree $1$ which generate the group $\mathrm{Aut}(\mathbb{P}^2_\mathbb{C})=\mathrm{PGL}(3;\mathbb{C})$. Let us give some examples of quadratic birational maps:
\begin{align*}
&\sigma\colon (x:y:z)\dashrightarrow (yz:xz:xy),
&&\rho\colon (x:y:z)\dashrightarrow (xy:z^2:yz),\\
&\tau\colon (x:y:z)\dashrightarrow (x^2:xy:y^2-xz).
\end{align*}

These three maps, which are involutions, play an important role in the description of the set of quadratic birational maps of $\mathbb{P}^2_\mathbb{C}$.

The birational maps of $\mathbb{P}^2_\mathbb{C}$ form a group denoted $\mathrm{Bir}(\mathbb{P}^2_\mathbb{C})$ and called \textbf{\textit{Cremona group}}. If $\phi$ is an element of~$\mathrm{Bir}(\mathbb{P}^2_\mathbb{C})$ then $\mathscr{O}(\phi)$ is the orbit of $\phi$ under the action of $\mathrm{Aut}(\mathbb{P}^2_\mathbb{C})\times\mathrm{Aut}(\mathbb{P}^2_\mathbb{C})$:
\[
\mathscr{O}(\phi)=\big\{\ell\phi \ell'\,\vert\,\ell,\,\ell'\in\mathrm{Aut}(\mathbb{P}^2_\mathbb{C})\big\}.
\]
A very old theorem, often called N\oe ther Theorem, says that any element of $\mathrm{Bir}(\mathbb{P}^2_\mathbb{C})$ can be written, up to the action of an automorphism of $\mathbb{P}^2_\mathbb{C}$, as a composition of quadratic birational maps (\cite{Castelnuovo}). In \cite[Chapters 1 \& 6]{CD} the structure of the set $\mathrm{Bir}_d$ (resp. $\mathrm{\mathring{B}ir}_d$) of birational maps of $\mathbb{P}^2_\mathbb{C}$ of degree $\leq d$ (resp. of degree $d$) has been studied when $d=2$ and $d=3$.

\begin{thm}[Corollary 1.10, Theorem 1.31, \cite{CD}]\label{Thm:decomp}
One has the following decomposition
\[
\mathrm{\mathring{B}ir}_2=\mathscr{O}(\sigma)\cup\mathscr{O}(\rho)\cup\mathscr{O}(\tau).
\]

Furthermore 
\[
\mathrm{Bir}_2=\overline{\mathscr{O}(\sigma)}
\]
where $\overline{\mathscr{O}(\sigma)}$ denotes the ordinary closure of $\mathscr{O}(\sigma)$, and
\[
\dim\mathscr{O}(\tau)=12,\quad\quad\dim\mathscr{O}(\rho)=13,\quad\quad\dim\mathscr{O}(\sigma)=14.
\]
\end{thm}

Note that there is a more precise description of $\mathrm{Bir}_2$ in \cite[Chapter 1]{CD}.

We will further do some computations with birational maps of degree $3$. Let us consider the following family of cubic birational maps: 
\[
\Phi_{a,b}\colon (x:y:z)\dashrightarrow \big(x(x^2+y^2+axy+bxz+yz):y(x^2+y^2+axy+bxz+yz):xyz\big)
\]
with $a$, $b\in\mathbb{C}$, $a^2\not=4$ and $2b\not\in\{a\pm\sqrt{a^2-4}\}$.
The structure of $\mathrm{Bir}_3$ is much more complicated than the structure
of $\mathrm{Bir}_2$ (\emph{see} \cite[Chapter 6]{CD}), nevertheless one has the following result.

\begin{thm}[Proposition 6.35, Theorem 6.38, \cite{CD}]
The closure of 
\[
\mathscr{X}=\big\{\mathscr{O}(\Phi_{a,b})\,\vert\, a,\, b\in\mathbb{C},\,a^2\not=4,\,2b\not\in\{a\pm\sqrt{a^2-4}\}\big\}
\]
in the set of rational maps of degree $3$ is an irreducible algebraic variety of dimension $18$.

Furthermore the closure of $\mathscr{X}$ in $\mathrm{\mathring{B}ir}_3$ is $\mathrm{\mathring{B}ir}_3$.
\end{thm}

In the sequel we will say that a $\Phi_{a,b}$ is a generic element of $\mathrm{Bir}_3$.

The "most degenerate model" \footnote{In the following sense: the exceptional
  locus of any element of $\mathrm{Bir}_3$ is a union of degree $6$ of conics
  and lines; the exceptional locus of $\Psi$ is reduced to a single line
  of multiplicity $6$.} is up to automorphisms of $\mathbb{P}^2_\mathbb{C}$
\[
\Psi\colon (x:y:z)\dashrightarrow \big(xz^2+y^3:yz^2:z^3\big).
\]

\subsection{About foliations}

\begin{defi}
Let $\mathcal{F}$ be a foliation (maybe singular) on a complex manifold $M$; the foliation~$\mathcal{F}$ is a \textbf{\textit{singular transversely projective}} one if there exists
\begin{itemize}
\item[$\mathfrak{a}$)] $\pi\colon P\to M$ a $\mathbb{P}^1$-bundle over $M$,
\item[$\mathfrak{b}$)] $\mathcal{G}$ a codimension one singular holomorphic foliation on $P$ transversal to the generic fibers of~$\pi$,
\item[$\mathfrak{c}$)] $\varsigma\colon M\to P$ a meromorphic section generically transverse to $\mathcal{G}$,
\end{itemize}
such that $\mathcal{F}=\varsigma^*\mathcal{G}$. 
\end{defi}

Let us give an other characterization of singular transversely projective
foliations. Let $\mathcal{F}$ be a foliation on $\mathbb{P}^2_\mathbb{C}$; assume that there exist three rational $1$-forms $\theta_0$, $\theta_1$ and $\theta_2$ on $\mathbb{P}^2_\mathbb{C}$ such that 
\begin{itemize}
\item[$\mathfrak{i}$)] $\mathcal{F}$ is described by $\theta_0$, {\it i.e.} $\mathcal{F}=\mathcal{F}_{\theta_0}$,

\item[$\mathfrak{ii}$)] the $\theta_i$'s form a $\mathfrak{sl}(2;\mathbb{C})$-triplet, that is
\begin{align*}
&\mathrm{d}\theta_0=\theta_0\wedge\theta_1, &&\mathrm{d}\theta_1=\theta_0\wedge\theta_2, && \mathrm{d}\theta_2=\theta_1\wedge\theta_2.
\end{align*}
\end{itemize}
Then $\mathcal{F}$ is a singular transversely projective foliation. To see it one considers the manifolds $M=\mathbb{P}^2_\mathbb{C}$, $P=\mathbb{P}^2_\mathbb{C}\times\mathbb{P}^1_\mathbb{C}$, the canonical projection $\pi\colon P\to M$, and the foliation $\mathcal{G}$ given by 
\[
\theta=\mathrm{d}z+\theta_0+z\theta_1+\frac{z^2}{2}\theta_2
\]
where $z$ is an affine coordinate of $\mathbb{P}^1_\mathbb{C}$; in that case the transverse section is $z=0$. When one can choose the $\theta_i$'s such that $\theta_1=\theta_2=0$ (resp. $\theta_2=0$) the foliation $\mathcal{F}$ is \textbf{\textit{defined by a closed $1$-form}} (resp. is \textbf{\textit{transversely affine}}). 

\medskip

Classical examples of singular transversely projective foliations are given by Riccati foliations.

\begin{defi}
A \textbf{\textit{Riccati equation}} is a differential equation of the following type
\[
\mathcal{E}_R\colon y'= a(x)y^2 +b(x)y+c(x)
\]
where $a$, $b$ and $c$ are meromorphic functions on an open subset $\mathcal{U}$ of $\mathbb{C}$. To the equation $\mathcal{E}_R$ one associates the meromorphic differential form
\[
\omega_{\mathcal{E}_R}=\mathrm{d}y-\big(a(x)y^2 +b(x)y+c(x)\big)\,\mathrm{d}x
\]
defined on $\mathcal{U}\times\mathbb{C}$. In fact $\omega_{\mathcal{E}_R}$ induces a foliation $\mathcal{F}_{\omega_{\mathcal{E}_R}}$ on $\mathcal{U}\times\mathbb{P}^1_\mathbb{C}$ that is transverse to the generic fiber of the projection $\mathcal{U}\times\mathbb{P}^1_\mathbb{C}\to\mathcal{U}$. One can check that 
\begin{align*}
&\theta_0=\omega_{\mathcal{E}_R}, &&\theta_1=-(2a(x)y+b(x))\,\mathrm{d}x, &&\theta_2=-2a(x)\,\mathrm{d}x
\end{align*}
is a $\mathfrak{sl}(2;\mathbb{C})$-triplet associated to the foliation $\mathcal{F}_{\omega_{\mathcal{E}_R}}$. 

We say that $\omega_{\mathcal{E}_R}$ is a \textbf{\textit{Riccati $1$-form}} and $\mathcal{F}_{\omega_{\mathcal{E}_R}}$ is a \textbf{\textit{Riccati foliation}}.
\end{defi}

Let $S$ be a ruled surface, that is a surface $S$ endowed with $f\colon S\to\mathcal{C}$, where $\mathcal{C}$ denotes a curve and $f^{-1}(c)\simeq~\mathbb{P}^1_\mathbb{C}$. Let us consider a singular foliation $\mathcal{F}$ on $S$ transverse to the generic fibers of $f$. The foliation $\mathcal{F}$ is transversely projective.

Recall that a foliation $\mathcal{F}$ on a surface $S$ is \textbf{\textit{radial}} at a point $m$ of $S$ if in local coordinates $(x,y)$ around~$m$ the foliation $\mathcal{F}$ is given by a holomorphic $1$-form of the following type 
\[
\omega=x\,\mathrm{d}y-y\,\mathrm{d}x+ \text{ h.o.t. }
\]

Let us denote by $\mathbb{F}(n;d)$ the set of foliations of degree $d$ on $\mathbb{P}^n_\mathbb{C}$ (\emph{see} \cite{CCD}). The following statement gives a criterion which asserts that an element of $\mathbb{F}(2;2)$ is transversely projective.

\begin{pro}
Let $\mathcal{F}\in\mathbb{F}(2;2)$ be a foliation of degree $2$ on $\mathbb{P}^2_\mathbb{C}$. If a singular point of $\mathcal{F}$ is radial, then~$\mathcal{F}$ is transversely projective.
\end{pro}

\begin{proof}
Assume that the singular point is the origin $0$ in the affine chart $z=1$, the foliation $\mathcal{F}$ is thus defined by a $1$-form of the following type
\[
\omega=x\,\mathrm{d}y-y\,\mathrm{d}x +q_1\,\mathrm{d}x+q_2\,\mathrm{d}y+q_3(x\,\mathrm{d}y-y\,\mathrm{d}x)
\]
where the $q_i$'s denote quadratic forms. Let us consider the complex projective plane $\mathbb{P}^2_\mathbb{C}$ blown up at the origin; this space is denoted by $\mathrm{Bl}(\mathbb{P}^2_\mathbb{C},0)$. Let 
\[
\pi\colon\mathrm{Bl}(\mathbb{P}^2_\mathbb{C},0)\to\mathbb{P}^2_\mathbb{C} 
\]
be the canonical projection. Then $\pi^*\mathcal{F}$ is transverse to the generic fibers of $\pi$, and in fact transverse to all the fibers excepted the strict transforms of the lines $xq_1+yq_2=0$. Hence the foliation $\pi^*\mathcal{F}$ is transversely projective; since this notion is invariant under the action of a birational map, $\mathcal{F}$ is transversely projective.
\end{proof}

\begin{rem}
The same argument can be involved for foliations of degree $2$ on $\mathbb{P}^2_\mathbb{C}$ having a singular point with zero $1$-jet.
\end{rem}

\begin{rem}\label{rem:deltaR}
The closure of the set $\Delta_R$ of foliations in $\mathbb{F}(2;2)$ having a radial singular point is irreducible, of codimension $2$ in $\mathbb{F}(2;2)$. Indeed $\Delta_R$ is the $\mathrm{Aut}(\mathbb{P}^2_\mathbb{C})$-orbit of the set 
\[
\big\{x\,\mathrm{d}y-y\,\mathrm{d}x+q_1\,\mathrm{d}x+q_2\,\mathrm{d}y+q_3(x\,\mathrm{d}y-y\,\mathrm{d}x)\,\big\vert\,q_i\text{ quadratic form}\big\};
\]
in fact it is easy to see that $\overline{\Delta_R}$ is the $\mathrm{Aut}(\mathbb{P}^2_\mathbb{C})$-orbit of 
\[
\big\{\lambda(x\,\mathrm{d}y-y\,\mathrm{d}x)+q_1\,\mathrm{d}x+q_2\,\mathrm{d}y+q_3(x\,\mathrm{d}y-y\,\mathrm{d}x)\,\big\vert\,\lambda\in\mathbb{C},\,q_i\text{ quadratic form}\big\}.
\]
In particular $\overline{\Delta_R}$ is an unirational set in $\mathbb{F}(2;2)$.
\end{rem}

\section{Proof of Theorem \ref{Thm:main}}\label{Sec:proofthmA}

We establish Theorem \ref{Thm:main} in two steps: we first look at foliations that have at least two singular points and then at foliations with exactly one singular point.

\subsection{Foliations of degree $2$ on $\mathbb{P}^2_\mathbb{C}$ with at least two singularities}

Any $\mathcal{F}\in\mathbb{F}(2;2)$ is described in homogeneous coordinates by a $1$-form $\omega$ that can be written 
\begin{equation}\label{eq:formenormalefeuilletage}
\omega=q_1yz\left(\frac{\mathrm{d}y}{y}-\frac{\mathrm{d}z}{z}\right)+q_2xz\left(\frac{\mathrm{d}z}{z}-\frac{\mathrm{d}x}{x}\right)+q_3xy\left(\frac{\mathrm{d}x}{x}-\frac{\mathrm{d}y}{y}\right)
\end{equation}
where 
\begin{align*}
&q_1=a_0x^2+a_1y^2+a_2z^2+a_3xy+a_4xz+a_5yz, \\
& q_2=b_0x^2+b_1y^2+b_2z^2+b_3xy+b_4xz+b_5yz, \\
&q_3=c_0x^2+c_1y^2+c_2z^2+c_3xy+c_4xz+c_5yz.&&
\end{align*}

\begin{pro}
For any $\mathcal{F}\in\mathbb{F}(2;2)$ with at least two distinct singularities there exists a quadratic birational map $\psi\in\mathscr{O}(\rho)$ such that $\deg\psi^*\mathcal{F}\leq 3$. 
\end{pro}

\begin{proof}
In homogeneous coordinates $\mathcal{F}$ is described by a $1$-form $\omega$ as in (\ref{eq:formenormalefeuilletage}).

Up to an automorphism of $\mathbb{P}^2_\mathbb{C}$ one can suppose that $(1:0:0)$ and $(0:1:0)$ are singular points of~$\mathcal{F}$, that is $a_1=b_0=c_0=c_1=0$. If $c_3\not=0$, resp. $c_3=0$ and $b_4\not=0$, resp. $c_3=b_4=0$, then let us consider the quadratic birational map~$\psi$ of~$\mathscr{O}(\rho)$ defined as follows
\[
\psi\colon (x:y:z)\dashrightarrow\Big(xy:z^2+\frac{b_3-c_4+\sqrt{(b_3-c_4)^2+4b_4c_3}}{2c_3}\,yz:yz\Big),
\]
resp.
\[
\psi\colon (x:y:z)\dashrightarrow\Big(xy:z^2+yz:-\frac{b_3-c_4}{b_4}\,yz\Big),
\]
resp. $\psi=\rho$.
 A direct computation shows that $\psi^*\omega=yz^2\omega'$ where $\omega'$ denotes a homogeneous $1$-form of degree~$4$. The foliation $\mathcal{F}'$ defined by $\omega'$ has degree at most $3$.
\end{proof}

\subsection{Foliations of degree $2$ on $\mathbb{P}^2_\mathbb{C}$ with exactly one singularity}

Such foliations have been classified:

\begin{thm}[\cite{CDGBM}]\label{Thm:class1sing}
Up to automorphisms of $\mathbb{P}^2_\mathbb{C}$ there are four foliations of degree~$2$ on $\mathbb{P}^2_\mathbb{C}$ having exactly one singularity. They are described in affine chart by the following $1$-forms:
\smallskip
\begin{itemize}
\item[$\bullet$] $\Omega_1 = x^2\,\mathrm{d}x + y^2(x\,\mathrm{d}y-y\,\mathrm{d}x),$
\smallskip
\item[$\bullet$] $\Omega_2=x^2\mathrm{d}x+(x+y^2)(x\,\mathrm{d}y-y\,\mathrm{d}x),$
\smallskip
\item[$\bullet$] $\Omega_3=xy\,\mathrm{d}x + (x^2 + y^2)(x\,\mathrm{d}y-y\,\mathrm{d}x),$
\smallskip
\item[$\bullet$] $\Omega_4=(x+y^2-x^2y)\,\mathrm{d}y+x(x+y^2)\,\mathrm{d}x.$
\end{itemize}
\end{thm}

\begin{pro}\label{Pro:mod}
There exists a quadratic birational map $\psi_1\in\mathscr{O}(\rho)$ such that $\deg\psi_1^*\mathcal{F}_{\Omega_1}=~2$; furthermore $\mathcal{F}_{\Omega_1}$ has a rational first integral and is non-primitive.

\smallskip

For $k=2$, $3$, there is no birational map $\varphi_k$ such that $\deg\varphi_k^*\mathcal{F}_{\Omega_k}=0$ but there is $\psi_k$ in $\mathscr{O}(\tau)$ such that $\deg\psi_k^*\mathcal{F}_{\Omega_k}=~1$. In particular $\mathcal{F}_{\Omega_2}$ and $\mathcal{F}_{\Omega_3}$ are non-primitive.

\smallskip

There is a quadratic birational map $\psi_4\in\mathscr{O}(\tau)$ such that $\deg\psi_4^*\mathcal{F}_{\Omega_4}=3$, and~$\mathcal{F}_{\Omega_4}$ is primitive.
\end{pro}

\begin{rem}
If $\phi=(x^2:xy:xz+y^2)$, then $\deg\phi^*\mathcal{F}_{\Omega_2}=\deg\phi^*\mathcal{F}_{\Omega_3}=2$. A contrario we will see later there is no quadratic birational map $\phi$ such that $\deg\phi^*\mathcal{F}_{\Omega_4}=2$ (\emph{see} Corollary~\ref{Cor:omega4}). 
\end{rem}

\begin{cor}
For any element $\mathcal{F}$ of $\mathbb{F}(2;2)$ with exactly one singularity there exists a quadratic birational map $\psi$ such that $\deg\psi^*\mathcal{F}\leq 3$.
\end{cor}

\begin{proof}[Proof of Proposition \ref{Pro:mod}]
The foliation $\mathcal{F}_{\Omega_1}$  is given in homogeneous coordinates by 
\[
\Omega'_1=(x^2z-y^3)\,\mathrm{d}x + xy^2\,\mathrm{d}y-x^3\,\mathrm{d}z;
\]
if $\psi_1\colon(x:y:z)\dashrightarrow(x^2:xy:yz)$ then 
\[
\psi_1^*\Omega'_1\wedge\big(y(2xz-y^2)\,\mathrm{d}x+x(y^2-xz)\,\mathrm{d}y-x^2y\,\mathrm{d}z\big)=0.
\]
The foliation $\mathcal{F}_{\Omega_1}$ has a rational first integral and is non-primitive, it is the image of a foliation of degree $0$ by a cubic birational map: 
\[
(x^3: x^2y: x^2z+y^3/3)^*\Omega'_1\wedge\big(z\,\mathrm{d}x-x\,\mathrm{d}z\big)=0.
\]

\smallskip

The foliation $\mathcal{F}_{\Omega_2}$ is described in homogeneous coordinates by 
\[
\Omega'_2=(x^2z-xyz-y^3)\,\mathrm{d}x +x(xz+y^2)\,\mathrm{d}y-x^3\,\mathrm{d}z;
\]
let us consider the birational map $\psi_2\colon(x:y:z)\dashrightarrow(x^2:xy:xz-2x^2-2xy-y^2)$ then 
\[
\psi_2^*\Omega'_2\wedge\big((xz-yz)\,\mathrm{d}x+xz\,\mathrm{d}y-x^2\,\mathrm{d}z\big)=0.
\]
One can verify that 
\[
\left(2+\frac{1}{x}+2\frac{y}{x}+\frac{y^2}{x^2}\right)\exp\left(-\frac{y}{x}\right)
\]
is a first integral of $\mathcal{F}_{\Omega_2}$; it is easy to see that $\mathcal{F}_{\Omega_2}$ has no rational first integral so there is no birational map $\varphi_2$ such that $\deg\varphi_2^*\mathcal{F}_{\Omega_2}=0$. 

\smallskip

The foliation $\mathcal{F}_{\Omega_3}$ is given in homogeneous coordinates by the $1$-form 
\[
\Omega'_3 =y(xz-x^2-y^2)\,\mathrm{d}x +x(x^2+y^2)\,\mathrm{d}y-x^2y\,\mathrm{d}z; 
\]
if $\psi_3\colon(x:y:z)\dashrightarrow(x^2:xy:xz+y^2/2)$
then 
\[
\psi_3^*\Omega'_3\wedge \big(y(z-x)\,\mathrm{d}x+x^2\,\mathrm{d}y-xy\,\mathrm{d}z\big)=0.
\]
The function
\[
\left(\frac{y}{x}\right)\exp\left(\frac{1}{2}\frac{y^2}{x^2}-\frac{1}{x}\right)
\]
is a first integral of $\mathcal{F}_{\Omega_3}$, and $\mathcal{F}_{\Omega_3}$ has no rational first integral so there is no birational map $\varphi_3$ such that $\deg\varphi_3^*\mathcal{F}_{\Omega_3}=~0$. 

\smallskip

Let us consider the birational map of $\mathbb{P}^2_\mathbb{C}$ given by
\[
\psi_4\colon(x:y:z)\dashrightarrow(-x^2:xy:y^2-xz)
\]
In homogeneous coordinates 
\[
\Omega'_4=x(xz+y^2)\,\mathrm{d}x+(xz^2+y^2z-x^2y)\,\mathrm{d}y+(xyz-y^3-x^3)\,\mathrm{d}z; 
\]
a direct computation shows that $\psi_4^*\Omega'_4\wedge\eta=0$ where
\[
\eta=(3y^3z-x^2y^2+x^3z-2xyz^2)\,\mathrm{d}x+(x^3y-4y^4-x^2z^2+3xy^2z)\,\mathrm{d}y
\]
\[
+x(2y^3-x^3-xyz)\,\mathrm{d}z.
\]
The foliation $\mathcal{F}_{\Omega_4}$ has no invariant algebraic curve so $\mathcal{F}_{\Omega_4}$ is not transversely projective (\cite[Proposition 1.3]{CDGBM}). In fact a foliation of degree $2$ without invariant algebraic curve is primitive (\emph{see} Introduction); as a consequence $\mathcal{F}_{\Omega_4}$ is a primitive foliation.
\end{proof}

\section{Numerical invariance}\label{Sec:numinv}

In the sequel num. inv. means numerically invariant.

\smallskip

In this section we determine the foliations $\mathcal{F}$ of $\mathbb{F}(2;2)$ num. inv. under the action of $\sigma$ (resp. $\rho$, resp.~$\tau$). Note that if $\phi$ is a birational map of $\mathbb{P}^2_\mathbb{C}$ and $\ell$ an element of $\mathrm{Aut}(\mathbb{P}^2_\mathbb{C})$ then $\deg(\phi\ell)^*\mathcal{F}=\deg\phi^*\mathcal{F}$; hence following Theorem \ref{Thm:decomp} we get the description of foliations num. inv. under the action of a quadratic birational map of $\mathbb{P}^2_\mathbb{C}$ by giving normal forms.

Recall that $\sigma$ is given in a fixed system of homogeneous coordinates $(x:y:z)$ by 
\[
\sigma\colon(x:y:z)\dashrightarrow(yz:xz:xy),
\]
and remark that $\sigma$ is invariant under conjugacy by elements of the group $\mathfrak{S}_3$ of standard permutations of coordinates.

\begin{lem}\label{Lem:sigma}
An element $\mathcal{F}$ of $\mathbb{F}(2;2)$ is num. inv. under the action of $\sigma$ if and only if it is given up to permutations of coordinates and standard affine charts by $1$-forms of the following type
\smallskip
\begin{itemize}
\item[$\bullet$] either $\omega_1=y\big(\kappa+\varepsilon y\big)\,\mathrm{d}x+\big(\beta x+\delta y+\alpha x^2+\gamma xy\big)\,\mathrm{d}y,$
\smallskip
\item[$\bullet$] or $\omega_2=\big(\delta+\beta y+\kappa y^2 \big)\,\mathrm{d}x+\big(\alpha+\varepsilon x+\gamma x^2\big)\,\mathrm{d}y,$
\end{itemize}
\smallskip
where $\alpha$, $\beta$, $\gamma$, $\delta$, $\varepsilon$, $\kappa$ $($resp. $\alpha$, $\beta$, $\gamma$, $\delta$, $\varepsilon$, $\kappa)$ are complex numbers such that $\deg\mathcal{F}_{\omega_1}=2$ $($resp. $\deg\mathcal{F}_{\omega_2}=2)$.
\end{lem}

\begin{proof}
The foliation $\mathcal{F}$ is defined by a homogeneous $1$-form $\omega$ of degree $3$.
The map $\sigma$ is an automorphism of $\mathbb{P}^2_\mathbb{C}\smallsetminus\{xyz=0\}$; hence if $\sigma^*\omega=P\omega'$, with $\omega'$ a $1$-form of degree $3$ and $P$ a homogeneous polynomial, then $P=x^iy^jz^k$ for some integers $i$, $j$, $k$ such that $i+j+k=4$. Up to permutation of coordinates it is sufficient to look at the four following cases: $P=x^4$, $P=x^3y$, $P=x^2y^2$ and $P=x^2yz$. Let us write~$\omega$ as in (\ref{eq:formenormalefeuilletage}). Computations show that $x^4$  (resp. $x^3y$) cannot divide $\sigma^*\omega$. If $P=x^2yz$, then $\sigma^*\omega=P\omega'$ if and only~if 
\[
c_0=b_0=a_2=b_2=a_1=c_1=b_4=c_3=0,\quad b_3=c_4
\]
that gives $\omega_1$.
Finally one has $\sigma^*\omega=x^2y^2\omega'$ if and only if 
\[
c_1=c_0=b_0=a_1=b_4=c_3=a_5=0, \quad b_3=c_4,\quad c_5=a_3;
\]
in that case we obtain $\omega_2$.
\end{proof}

\begin{pro}
A foliation $\mathcal{F}\in\mathbb{F}(2;2)$ num. inv. under the action of an element of $\mathscr{O}(\sigma)$ is $\mathrm{Aut}(\mathbb{P}^2_\mathbb{C})$-conjugate either to a foliation of type $\mathcal{F}_{\omega_1}$, or to a foliation of type $\mathcal{F}_{\omega_2}$. In particular it is transversely projective. 
\end{pro}

\begin{proof}
Let $\phi$ be an element of $\mathscr{O}(\sigma)$ such that $\deg\phi^*\mathcal{F}=2$; the map $\phi$ can be written $\ell_1\sigma\ell_2$ where~$\ell_1$ and $\ell_2$ denote automorphisms of $\mathbb{P}^2_\mathbb{C}$. By assumption the degree of $(\ell_1\sigma\ell_2)^*\mathcal{F}=\ell_2^*(\sigma^*(\ell_1^*\mathcal{F}))$ is~$2$. Hence $\deg \sigma^*(\ell_1^*\mathcal{F})=2$ and the foliation~$\ell_1^*\mathcal{F}$ is num. inv. under the action of $\sigma$. Since $\ell_1^*\mathcal{F}$ and $\mathcal{F}$ are conjugate and since the notion of transversal projectivity is invariant by conjugacy it is sufficient to establish the statement for $\phi=\sigma$. The proposition thus follows from the fact that $1$-forms of Lemma \ref{Lem:sigma} are Riccati ones (up to multiplication).
\end{proof}

\begin{rem}
For generic values of parameters $\alpha$, $\beta$, $\gamma$, $\delta$, $\varepsilon$, $\kappa$ a foliation of type~$\mathcal{F}_{\omega_1}$ given by the corresponding form $\omega_1$ is not given by a closed meromorphic $1$-form. This can be seen by studying the holonomy group of $\mathcal{F}_{\omega_1}$ that can be identified with a subgroup of $\mathrm{PGL}(2;\mathbb{C})$ generated by two elements $f$ and $g$. For generic values of the parameters $f$ and $g$ are also generic, in particular the group $\langle f,\,g\rangle$ is free. When $\mathcal{F}_{\omega_1}$ is given by a closed $1$-form, then the holonomy group is an abelian one.

Remark that a contrario the foliations given by $1$-forms of type $\omega_2$ are given by a closed meromorphic $1$-form.
\end{rem}

\begin{rem}
Let $\Delta_i$ denote the closure of the set of elements of $\mathbb{F}(2;2)$ conjugate to a foliation of type~$\mathcal{F}_{\omega_i}$. 

In the affine chart $x=1$ an element of type $\mathcal{F}_{\omega_1}$ is radial at $(0,0)$ as soon as $\alpha\not=0$. Since $\Delta_R$ is closed, the inclusion $\Delta_1\subset\Delta_R$ holds.

If the components of $\omega_2$ are not constant, then an element of type $\mathcal{F}_{\omega_2}$ has a singular point in $\mathbb{C}^2$, and up to an ad-hoc translation $\mathcal{F}_{\omega_2}$ is an element of type $\mathcal{F}_{\omega_1}$. As $\Delta_1$ is closed, one has $\Delta_2\subset\Delta_1$.
\end{rem}

\begin{rem}
The notion of num. inv. is not related to the dynamic of the map (\emph{see}~\cite{CF} for example): the foliations num. inv. by the involution $\sigma$ ("without dynamic") are conjugate to the foliations num. inv. by $A\sigma$, $A\in\mathrm{Aut}(\mathbb{P}^2_\mathbb{C})$, which has a rich dynamic for generic $A$.
\end{rem}

The foliations of $\mathbb{F}(2;2)$ invariant by $\sigma$ are particular cases of num. inv. foliations:

\begin{pro}\label{Pro:invsigma}
An element of $\mathbb{F}(2;2)$ invariant by $\sigma$ is given up to permutations of coordinates and affine charts
\smallskip
\begin{itemize}
\item[$\bullet$] either by $y(1+y)\,\mathrm{d}x+(\beta x+\alpha y+\alpha x^2+\beta xy)\,\mathrm{d}y,$
\smallskip
\item[$\bullet$] or by $y(1-y)\,\mathrm{d}x+(\beta x-\alpha y+\alpha x^2-\beta xy)\,\mathrm{d}y,$
\smallskip
\item[$\bullet$] or by $y\,\mathrm{d}x+(\alpha+\varepsilon x+\alpha x^2)\,\mathrm{d}y,$
\end{itemize}
\smallskip
where the parameters are complex numbers such that the degree of the associated foliations is $2$.
\end{pro}

\begin{proof}
With the notations of Lemma \ref{Lem:sigma} one has 
\[
\sigma^*\omega_1=-y(\varepsilon+\kappa y)\,\mathrm{d}x-(\gamma x+\alpha y+\delta x^2+\beta xy)\,\mathrm{d}y;
\]
thus $\sigma^*\omega_1\wedge\omega_1=0$ if and only if either $\gamma = \beta$, $\delta = \alpha$, $\varepsilon = \kappa$, or $\gamma = -\beta$, $\delta = -\alpha$, $\varepsilon = -\kappa$. 

One has $\sigma^*\omega_2=-(\kappa+\beta y+\delta y^2)\,\mathrm{d}x-(\gamma+\varepsilon x+\alpha x^2)\,\mathrm{d}y$, and $\omega_2\wedge\sigma^*\omega_2=0$ if and only if $\gamma = \alpha$, $\delta = 0$ and $\kappa = 0$.
\end{proof}

\begin{rem}
The foliations associated to the two first $1$-forms with parameters $\alpha$, $\beta$ of Proposition \ref{Pro:invsigma} are conjugate by the automorphism $(x,y)\mapsto(x,-y)$.
\end{rem}

\begin{lem}\label{Lem:rho}
A foliation $\mathcal{F}\in\mathbb{F}(2;2)$ is num. inv. under the action of $\rho$ if and only if~$\mathcal{F}$ is given in affine chart 
\smallskip
\begin{itemize}
\item[$\bullet$] either by $\omega_3=y\big(\kappa+\varepsilon y+\lambda y^2\big)\,\mathrm{d}x+\big(\beta+\kappa x+\delta y+\gamma xy+\alpha y^2-\lambda xy^2\big)\,\mathrm{d}y,$

\smallskip

\item[$\bullet$] or by $\omega_4=y\big(\mu+\delta x+\gamma y+\varepsilon xy\big)\,\mathrm{d}x+\big(\alpha+\beta x+\lambda y+\delta x^2+\kappa xy-\varepsilon x^2y\big)\,\mathrm{d}y,$

\smallskip

\item[$\bullet$] or by $\omega_5=\big(\lambda+\gamma y+\kappa xy+\varepsilon y^2\big)\,\mathrm{d}x+\big(\beta+\delta x+\alpha x^2\big)\,\mathrm{d}y,$
\end{itemize}
\smallskip
where the parameters of $\omega_i$ are such that $\deg\mathcal{F}_{\omega_i}=2$.
\end{lem}

\begin{proof}
Let us take the notations of the proof of Lemma \ref{Lem:sigma}.
The map $\rho$ is an automorphism of $\mathbb{P}^2_\mathbb{C}\smallsetminus\{yz=~0\}$; therefore if $\rho^*\omega=P\omega'$ with $\omega'$ a $1$-form of degree $3$ and $P$ a homogeneous polynomial then $P=y^jz^k$ for some integers $j$, $k$ such that $j+k=4$. We have to look at the five following cases: $P=z^4$, $P=yz^3$, $P=y^2z^2$, $P=y^3z$ and $P=y^4$. 
Computations show that $y^4$ (resp. $y^3z$) cannot divide $\rho^*\omega$. If $P=z^4$ then $\rho^*\omega=P\omega'$ if and only if 
\[
c_0=b_0=c_3=b_4=b_2=0,\quad a_0=c_4,\quad b_3=c_4, \quad a_4=2c_2-b_5;
\]
this gives the first case $\omega_3$.
The equality $\rho^*\omega=yz^3\omega'$ holds if and only if 
\[
b_0=c_0=b_4=c_1=a_1=b_2=0,\quad a_0=2c_4-b_3
\]
and we obtain $\omega_4$.
Finally one has $\rho^*\omega=y^2z^2\omega'$ if and only if 
\[
c_1=b_0=c_3=a_5=a_1=c_0=b_4=0, \quad c_5=a_3
\]
which corresponds to $\omega_5$.
\end{proof}

\begin{pro}
The foliations of type $\mathcal{F}_{\omega_3}$ and $\mathcal{F}_{\omega_5}$ are transversely projective. In fact the $\mathcal{F}_{\omega_3}$ are transversely affine, and the $\mathcal{F}_{\omega_5}$ are Riccati ones.
\end{pro}

\begin{proof}
A foliation of type $\mathcal{F}_{\omega_3}$ is described by the $1$-form
\[
\theta_0=\mathrm{d}x-\frac{\big(\beta+\delta y+\alpha y^2\big)+\big(\kappa +\gamma y-\lambda y^2\big)x}{y\big(\kappa+\varepsilon y+\lambda y^2\big)}\,\mathrm{d}y
\]
and it is transversely affine; to see it consider the $\mathfrak{sl}(2;\mathbb{C})$-triplet 
\[
\theta_0, \quad\quad\quad\theta_1=\frac{\kappa+\gamma y-\lambda y^2}{y(\kappa+\varepsilon y+\lambda y^2)}\,\mathrm{d}y,\quad\quad\quad\theta_2=0.
\]

\smallskip

A foliation of type $\mathcal{F}_{\omega_5}$ is given by 
\[
\mathrm{d}y+\frac{\lambda+(\gamma+\kappa)y+\varepsilon y^2}{\beta+\delta x+\alpha x^2}\,\mathrm{d}x
\]
and thus is a Riccati foliation. In fact the fibration $x/z=$ constant is transverse to~$\mathcal{F}_{\omega_5}$ that generically has three invariant lines. 
\end{proof}

We don't know if the $\mathcal{F}_{\omega_4}$ are transversely projective. For generic values of the para\-meters a foliation of type $\mathcal{F}_{\omega_4}$ hasn't meromorphic uniform first integral in the affine chart $y=1$. Thus if $\mathcal{F}_{\omega_4}$ is transversely projective then it must have an invariant algebraic curve different from $y=0$ (\emph{see} \cite{CLNLPT}). We don't know if it is the case. A foliation of degree $2$ is conjugate to a generic $\mathcal{F}_{\omega_4}$ (by an automorphism of~$\mathbb{P}^2_\mathbb{C}$) if and only if it has an invariant line (say $y=0$) with a singular point (say $0$) and local model $2x\,\mathrm{d}y-y\,\mathrm{d}x$. The closure of the set of such foliations has codimension~$2$. Note that the three families $\mathcal{F}_{\omega_3}$, $\mathcal{F}_{\omega_4}$ and $\mathcal{F}_{\omega_5}$ have non trivial intersection. The set $\overline{\{\mathcal{F}_{\omega_4}\}}$ contains many interesting elements such that the famous Euler foliation given by $y^2\,\mathrm{d}x+(y-x)\,\mathrm{d}y$; this foliation is transversely affine but is not given by a closed rational $1$-form.

\begin{pro}
A foliation $\mathcal{F}\in\mathbb{F}(2;2)$ num. inv. under the action of an element of $\mathscr{O}(\rho)$ is conjugate to a foliation either of type $\mathcal{F}_{\omega_3}$, or of type $\mathcal{F}_{\omega_4}$, or of type~$\mathcal{F}_{\omega_5}$.
\end{pro}

Let us look at special num. inv. foliations, those invariant by $\rho$.

\begin{pro}
An element of $\mathbb{F}(2;2)$ invariant by $\rho$ is given by a $1$-form of one of the following type
\smallskip
\begin{itemize}
\item[$\bullet$] $y (1-y)\,\mathrm{d}x+(\beta+x)\mathrm{d}y,$
\smallskip
\item[$\bullet$] $y^2\,\mathrm{d}x+(-1+y)\mathrm{d}y,$
\smallskip
\item[$\bullet$] $y(1-y)(\gamma+\delta x)\,\mathrm{d}x+(1+y)(\alpha+\beta x+\delta x^2)\,\mathrm{d}y,$
\smallskip
\item[$\bullet$] $y(1+y)(\gamma+\delta x)\,\mathrm{d}x+(1-y)(\alpha+\beta x+\delta x^2)\,\mathrm{d}y,$
\smallskip
\item[$\bullet$] $(1-y^2)\,\mathrm{d}x+(\beta+\delta x+\alpha x^2)\,\mathrm{d}y,$
\end{itemize}
\smallskip
where the parameters are complex numbers such that the degree of the associated foliations is $2$.
\end{pro}

\begin{cor}
An element of $\mathbb{F}(2;2)$ invariant by $\rho$ is defined by a closed $1$-form. 
\end{cor}

\begin{rem}
The third and fourth cases with parameters $\alpha$, $\beta$, $\gamma$, $\delta$ are conjugate by the automorphism $(x,y)\mapsto(x,-y)$.
\end{rem}

From Lemmas \ref{Lem:sigma} and \ref{Lem:rho} one gets the following statement.

\begin{pro}\label{Pro:numinvcurve}
A foliation num. inv. by an element of $\mathscr{O}(\phi)$, with $\phi=\sigma$, $\rho$, preserves an algebraic curve.
\end{pro}

\begin{cor}\label{Cor:omega4}
There is no quadratic birational map $\phi$ of the complex projective plane such that $\deg\phi^*\mathcal{F}_{\Omega_4}=~2$.
\end{cor}

\begin{proof}
The foliation $\mathcal{F}_{\Omega_4}$ has no invariant algebraic curve (\cite[Proposition 1.3]{CDGBM}); according to Proposition~\ref{Pro:numinvcurve} it is thus sufficient to show that there is no birational map $\phi\in\mathscr{O}(\tau)$ such that $\deg\phi^*\mathcal{F}_{\Omega_4}=2$ that can be established with  a direct and tedious computation.
\end{proof}

\begin{rem}
The map $\rho$ can be written $\ell_1\sigma \ell_2\sigma \ell_3$ with 
\begin{align*}
& \ell_1=(z-y:y-x:y), &&\ell_2=(y+z:z:x), &&\ell_3=(x+z:y-z:z).
\end{align*}
We are interested by the "intermediate" degrees of a numerically invariant foliation~$\mathcal{F}$, that is the sequence $\deg\mathcal{F}$, $\deg(\ell_1\sigma)^*\mathcal{F}$, $\deg(\ell_1\sigma\ell_2\sigma\ell_3)^*\mathcal{F}=\deg\mathcal{F}$. A tedious computation shows that for generic values of the parameters the sequence is $2$, $5$, $2$. We schematize this fact by the diagram
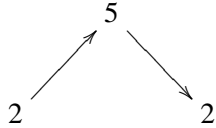
\[
\xymatrix{& 5\ar[dr] &\\
2\ar[ur] & &2}
\]
\end{rem}

A similar argument to Lemma \ref{Lem:sigma} yields to the following result.

\begin{lem}\label{Lem:tau}
An element $\mathcal{F}$ of $\mathbb{F}(2;2)$ is num. inv. under the action of $\tau$ if and only if $\mathcal{F}$ is given in affine chart by a $1$-form of type
\begin{eqnarray*}
\omega_6&=&\big(-\delta x+\alpha y-\varepsilon x^2+\theta xy+\beta y^2+\kappa x^2y+\mu xy^2+\lambda y^3\big)\,\mathrm{d}x\\
&&\hspace{1cm}+\big(-3\alpha x+\xi x^2+2(\delta-\beta)xy+\alpha y^2-\kappa x^3-\mu x^2y-\lambda xy^2\big)\,\mathrm{d}y
\end{eqnarray*}
where the parameters are such that $\deg\mathcal{F}_{\omega_6}=2$.
\end{lem}

We don't know the qualitative description of foliations of type $\mathcal{F}_{\omega_6}$. For example we don't know if the $\mathcal{F}_{\omega_6}$ are transversely projective. If it is the case, this implies the existence of invariant algebraic curves, and that fact is unknown.

\begin{pro}
A foliation $\mathcal{F}\in \mathbb{F}(2;2)$ num. inv. under the action of an element of $\mathscr{O}(\tau)$ is conjugate to $\mathcal{F}_{\omega_6}$ for suitable values of the parameters.
\end{pro}

Let us describe some special num. inv. foliations under the action of $\tau$, those invariant by $\tau$.

\begin{pro}
An element of $\mathbb{F}(2;2)$ invariant by $\tau$ is given 
\smallskip
\begin{itemize}
\item[$\bullet$] either by 
\begin{small}
\[
\big(-\varepsilon x^2+\theta xy+\beta y^2+\varepsilon xy^2-(\frac{\xi}{2}+\theta) y^3\big)\,\mathrm{d}x+x\big(\xi x-2\beta y-\varepsilon xy+(\frac{\xi}{2}+\theta) y^2\big)\,\mathrm{d}y,
\] 
\end{small}
\item[$\bullet$] or by
\begin{small} 
\[
\big(-\delta x+\alpha y+\frac{3}{2}\delta y^2+\kappa x^2y+\mu xy^2+\lambda y^3\big)\,\mathrm{d}x-\big(3\alpha x+\delta xy-\alpha y^2+\kappa x^3+\mu x^2y+\lambda xy^2\big)\,\mathrm{d}y,
\] 
\end{small}
\end{itemize}
where the parameters are complex numbers such that the degree of the associated foliations is $2$.

The foliations associated to the first $1$-form are transversely affine.
\end{pro}

\begin{proof}
The $1$-jet at the origin of the $1$-form 
\begin{small}
\[
\omega=\big(-\varepsilon x^2+\theta xy+\beta y^2+\varepsilon xy^2-(\frac{\xi}{2}+\theta) y^3\big)\,\mathrm{d}x+x\big(\xi x-2\beta y-\varepsilon xy+(\frac{\xi}{2}+\theta) y^2\big)\,\mathrm{d}y
\]
\end{small}
is zero so after one blow-up $\mathcal{F}_{\omega}$ is transverse to the generic fiber of the Hopf fibration; furthermore as the exceptional divisor is invariant, $\mathcal{F}_\omega$ is transversely affine.
\end{proof}

\begin{rem}
The map $\tau$ can be written $\ell_1\sigma \ell_2\sigma \ell_3\sigma\ell_2\sigma\ell_4$ with 
\begin{align*}
&\ell_1=(x-y:x-2y:-x+y-z), &&\ell_2=(x+z:x:y), \\
&\ell_3=(-y:x-3y+z:x), && \ell_4=(y-x:z-2x:2x-y).
\end{align*}
Let us consider a foliation $\mathcal{F}$ num. inv. under the action of $\tau$; set $\mathcal{F}'=\ell_1^*\mathcal{F}$. We compute the intermediate degrees: 
\begin{align*}
&\deg\sigma^*\mathcal{F}'=5, &&\deg(\sigma\ell_2\sigma)^*\mathcal{F}'=4, &&\deg(\sigma \ell_3\sigma\ell_2\sigma)^*\mathcal{F}'=5.
\end{align*}
To summarize:
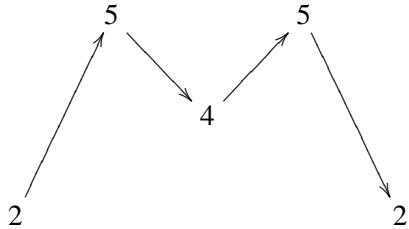
\[
\xymatrix{
& 5\ar[dr] &  &5\ar[ddr] &\\
& & 4\ar[ur] & &\\
2\ar[uur] & & & &2}
\]
\end{rem}

\section{Higher degree}\label{Sec:higherdegree}

We will now focus on similar questions but with cubic birational maps of $\mathbb{P}^2_\mathbb{C}$ and elements of $\mathbb{F}(2;2)$. The generic model of such birational maps is: 
\begin{small}
\[
\Phi_{a,b}\colon (x:y:z)\dashrightarrow \big(x(x^2+y^2+axy+bxz+yz):y(x^2+y^2+axy+bxz+yz):xyz\big)
\]
\end{small}
with $a$, $b\in\mathbb{C}$, $a^2\not=4$ and $2b\not\in\{a\pm\sqrt{a^2-4}\}$.

\begin{lem}\label{Lem:transfbircubgen}
An element $\mathcal{F}$ of $\mathbb{F}(2;2)$ is num. inv. under the action of $\Phi_{a,b}$ if and only if $\mathcal{F}$ is given in affine chart 
\smallskip
\begin{itemize}
\item[$\bullet$] either by $\omega_7=y(\alpha+\gamma y)\,\mathrm{d}x-x(\alpha +\kappa x)\,\mathrm{d}y,$
\smallskip
\item[$\bullet$] or by 
\[
\omega_8=b\big(b^2-ab+1+(a-2b)y+y^2\big)\,\mathrm{d}x+\big((b^2-ab+1)+(ab-2)x+x^2\big)\,\mathrm{d}y,
\]
\end{itemize}
\smallskip
where the parameters are such that $\deg\mathcal{F}_{\omega_7}=\deg\mathcal{F}_{\omega_8}=2$.
\end{lem}

\begin{rem}
Remark that the foliations $\mathcal{F}_{\omega_7}$ do not depend on the parameters of~$\Phi_{a,b}$, that is, the $\mathcal{F}_{\omega_7}$ are num. inv. by all $\Phi_{a,b}$, whereas the $\mathcal{F}_{\omega_8}$ only depend on~$a$ and $b$.

Furthermore $\mathcal{F}_{\omega_7}$ is num. inv. by $\sigma$ and $\rho$.
\end{rem}

\begin{pro}\label{Pro:transfbircubgen}
Any $\mathcal{F}\in\mathbb{F}(2;2)$ num. inv. under the action of $\Phi_{a,b}$, and more generally any $\mathcal{F}\in\mathbb{F}(2;2)$ num. inv. under the action of a generic cubic birational map of $\mathbb{P}^2_\mathbb{C}$, satisfies the following properties:
\smallskip
\begin{itemize}
\item[$\bullet$] $\mathcal{F}$ is given by a rational closed $1$-form;
\smallskip
\item[$\bullet$] $\mathcal{F}$ is non-primitive.
\end{itemize}
\end{pro}

\begin{proof}
Let us establish the properties for $\mathcal{F}_{\omega_7}$; remark that $\mathcal{F}_{\omega_7}$ is given by
\[
\frac{\mathrm{d}x}{x(\alpha+\kappa x)}-\frac{\mathrm{d}y}{y(\alpha+\gamma y)}
\]
which is a closed rational $1$-form (remark that $(\vert\alpha\vert+\vert\kappa\vert)(\vert\alpha\vert+\vert\gamma\vert)\not=0$ since $\deg\mathcal{F}_{\omega_7}=2$). The foliation $\mathcal{F}_{\omega_7}$ is non-primitive: indeed one has 
\[
\sigma^*\omega_7=\frac{1}{x^2y^2}\Big((\alpha x+\kappa)\,\mathrm{d}x-(\alpha y+\gamma)\,\mathrm{d}y\Big)
\]
that defines a foliation of degree $0$.
\smallskip

The idea and result are the same for the foliations $\mathcal{F}_{\omega_8}$ (except that it gives a birational map $\phi$ such that $\deg\phi^*\mathcal{F}_{\omega_8}=~1$).
\end{proof}

Let us consider an element $\mathcal{F}$ of $\mathbb{F}(2;2)$ num. inv. under the action of a birational map of degree $\geq 3$; is~$\mathcal{F}$ defined by a closed $1$-form ?

\begin{rem}
The foliations $\mathcal{F}_{\omega_7}$ are contained in the orbit of the foliation $\mathcal{F}_{\eta'}$.
\end{rem}

\begin{rem}
Any map $\Phi_{a,b}$ can be written $\ell_1\sigma \ell_2\sigma \ell_3$ with
\[
\ell_2=(a_0y+a_1z:a_2y+a_3z:a_4x+a_5y+a_6z)
\]
(\emph{see} \cite[proof of Proposition 6.36]{CD}). Let us consider the birational map $\xi=\sigma \ell_2\sigma$ with $$\ell_2=(a_0y+a_1z:a_2y+a_3z:a_4x+a_5y+a_6z)\in\mathrm{Aut}(\mathbb{P}^2_\mathbb{C}).$$ As in Lemma \ref{Lem:transfbircubgen} there are two families of foliations $\mathscr{F}_1$, $\mathscr{F}_2$ of degree $2$, one that does not depend on the parameters of $\xi$ and the other one depending only on the parameters of $\xi$, such that $\xi^*\mathscr{F}_1$ and $\xi^*\mathscr{F}_2$ are of degree $2$. One question is the following: what is the intermediate degree ? A computation shows that for generic parameters $\deg\sigma^*\mathscr{F}_1=4$ and that $\deg\sigma^*\mathscr{F}_2=2$. This implies in particular that~$\mathcal{F}_{\omega_8}$ is num. inv. under the action of $\sigma$.
For $\mathscr{F}_1$ and $\mathcal{F}_{\omega_7}$ one has
\[
\xymatrix{& 4\ar[dr]&\\
2\ar[ur] & & 2}
\]
and for $\mathscr{F}_2$ and $\mathcal{F}_{\omega_8}$
\[
\xymatrix{2\ar[r]& 2\ar[r] & 2}
\]
\end{rem}

Let us now consider the "most degenerate" cubic birational map 
\[
\Psi\colon (x:y:z)\dashrightarrow \big(xz^2+y^3:yz^2:z^3\big).
\]

\begin{lem}\label{Lem:transfbircubdegen}
An element $\mathcal{F}$ of $\mathbb{F}(2;2)$ is num. inv. under the action of $\Psi$ if and only if $\mathcal{F}$ is given in affine chart by
\[
\omega_9=\big(-\alpha+\beta y+\gamma y^2\big)\,\mathrm{d}x+\big(\varepsilon-3\beta x+\kappa y-3\gamma xy+\lambda y^2\big)\mathrm{d}y
\]
where the parameters are such that $\deg\mathcal{F}_{\omega_9}=2$.
In particular $\mathcal{F}$ is transversely affine.
\end{lem}

\begin{rem}
The map $\psi$ can be written $\ell_1\sigma \ell_2\sigma \ell_3\sigma \ell_4\sigma \ell_5\sigma\ell_4\sigma \ell_6\sigma \ell_2\sigma\ell_7$ with 
\begin{align*}
& \ell_1=(z-y:y:y-x), &&\ell_2=(y+z: z: x), && \ell_3=(-z:-y:x-y),\\
& \ell_4=(x+z: x: y), &&\ell_5=(-y: x-3y+z: x), && \ell_6=(-x: -y-z: x+y),\\
&\ell_7=(x+y:z-y:y). && &&
\end{align*}

As previously we consider the problem of the intermediate degrees; if $\mathcal{F}'=\ell_1^*\mathcal{F}$, a computation shows that for generic parameters
\[
\deg\sigma^*\mathscr{F}'=4, \quad\deg(\sigma \ell_2\sigma )^*\mathscr{F}'=3, \quad\deg(\sigma \ell_2\sigma \ell_3\sigma)^*\mathscr{F}'=5, 
\]
\[
\deg(\sigma \ell_2\sigma \ell_3\sigma \ell_4\sigma)^*\mathscr{F}'=3, \quad \deg(\sigma \ell_2\sigma \ell_3\sigma \ell_4\sigma \ell_5\sigma)^*\mathscr{F}'=5,
\]
\[
\deg(\sigma \ell_2\sigma \ell_3\sigma \ell_4\sigma \ell_5\sigma\ell_4\sigma )^*\mathscr{F}'=3, \quad\deg(\sigma \ell_2\sigma \ell_3\sigma \ell_4\sigma \ell_5\sigma\ell_4\sigma \ell_6\sigma)^*\mathscr{F}'=4,
\]
that is 
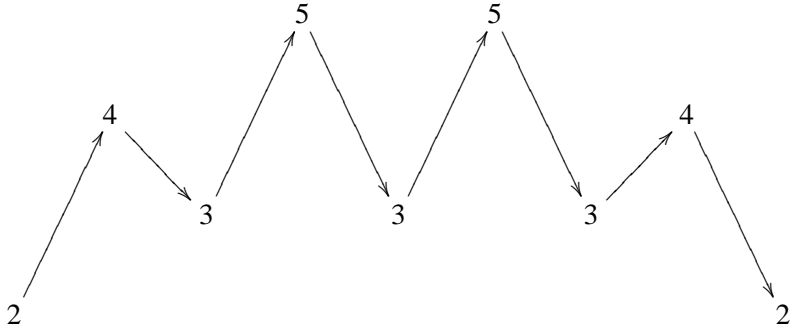
\[
\xymatrix{
 &  & & 5\ar[ddr] & & 5\ar[ddr] & & & \\
 & 4\ar[dr] & & & & & & 4\ar[ddr]& \\
 &  & 3\ar[uur] & & 3\ar[uur]& & 3\ar[ur]& & \\
2\ar[uur] & & & & & & & & 2}
\]
\end{rem}

We have not studied the quadratic foliations numerically invariant by (any) cubic birational transformation. It is reasonable to think that such foliations are transversely projective.

\bibliographystyle{plain}
\bibliography{bibliofeuilbir}

\end{document}